\documentclass[a4paper, 12pt]{amsart}
\usepackage{amsfonts, amsthm, amssymb, amsmath, stmaryrd}
\usepackage{mathrsfs,array}
\usepackage{xy}
\usepackage{verbatim}
\usepackage{amscd} 
\usepackage{tikz}
\usepackage{tikz-cd}
\usetikzlibrary{matrix,arrows,decorations.pathmorphing}
\usepackage{textcomp}
\input xy
\xyoption{all}
\usepackage[unicode]{hyperref}

\numberwithin{equation}{subsection}

\newtheorem{thm}{Theorem}[section]

\newtheorem{cor}[thm]{Corollary}

\newtheorem{lem}[thm]{Lemma} 
\newtheorem{prop}[thm]{Proposition}
 \theoremstyle{definition}
 \theoremstyle{definition}
\newtheorem{defn}[thm]{Definition} \theoremstyle{remark}
\newtheorem{rem}[thm]{\bf Remark}

\newtheorem{exa}[thm]{\bf Example}

\DeclareMathOperator{\Hom}{Hom}
\DeclareMathOperator{\End}{End}

\DeclareMathOperator{\Spec}{Spec}
\DeclareMathOperator{\Spm}{Spm}
\DeclareMathOperator{\Spf}{Spf}

\DeclareMathOperator{\Gal}{Gal}

\DeclareMathOperator{\Frob}{Frob}

\def\rig{\mathrm{rig}}

\def\HT{\mathrm{HT}}

\def\ps{\mathrm{ps}}

\def\an{\mathrm{an}}

\newcommand{\Z}{\mathbb{Z}}
\newcommand{\F}{\mathbb{F}}
\newcommand{\Q}{\mathbb{Q}}
\newcommand{\R}{\mathbb{R}}
\newcommand{\T}{\mathbb{T}}
\newcommand{\A}{\mathbb{A}}
\newcommand{\bC}{\mathbb{C}}

\newcommand{\GL}{\mathrm{GL}}

\newcommand{\Fl}{{\mathscr{F}\!\ell}}

\newcommand{\cO}{\mathcal{O}}

\newcommand{\km}{\mathfrak{m}}

\newcommand{\sC}{\mathscr{C}}

\newcommand{\RNum}[1]{\uppercase\expandafter{\romannumeral #1\relax}}

\newcommand{\overbar}[1]{\mkern 1.5mu\overline{\mkern-1.5mu#1\mkern-1.5mu}\mkern 1.5mu}

\begin{document}

\title{A note on some $p$-adic analytic Hecke actions}
\author{Lue Pan}
\address{Department of Mathematics, Princeton University, Fine Hall, Washington Road,
Princeton NJ 08544}
\email{lpan@princeton.edu}
\begin{abstract} 
We show that the action of Hecke operators away from $p$ on the space of ($p$-adic) overconvergent modular forms is ($p$-adically) locally analytic in a certain sense. As a corollary, the action of the Hecke algebra can be extended naturally to an action of rigid functions on its generic fiber. This directly determines the Hodge-Tate-Sen weights of Galois representation associated to an overconvergent eigenform and confirms a conjecture of Gouv\^{e}a.
\end{abstract}

\maketitle


\section{Introduction}
The notion of $p$-adic modular forms was introduced by Serre in the study of congruences between modular forms. It is well-known that to get a better spectral theory of the $U_p$-operator, one should consider the subspace of \textit{overconvergent} modular forms, on which $U_p$ acts completely continuously. In this short note, we will show that Hecke operators away from $p$ also have a better convergence when acting on overconvergent modular forms. As a consequence, we deduce that the action of the (big) Hecke algebra $\T$ naturally extends to an action of the rigid functions on its generic fiber (denoted by $\T^{\rig}$ by some people). Since having a Hodge-Tate-Sen weight $0$ is a Zariski-closed property on $\Spec\T^{\rig}$, the density of classical points implies directly that 

\begin{thm}[Corollary \ref{HTS}]
The two dimensional semi-simple Galois representation associated to an overconvergent eigenform of weight $k\in\Z$ has Hodge-Tate-Sen weights $0,k-1$.
\end{thm}
This confirms a conjecture of Gouv\^{e}a \cite[Conjecture 4]{Gou88}. We remark that this result was recently obtained by myself in \cite{Pan20} and by Sean Howe independently in \cite{Howe20} (when $k\neq 1$), by relating overconvergent modular forms with completed cohomology. Our method here is more straightforward. Hopefully it will be clear to the readers that the argument can be easily generalized to other contexts.

This note is organized as follows. We will first introduce a class of actions of algebras on a $p$-adic Banach space called \textit{locally analytic action} and give several (simple) examples. Then using fake-Hasse invariants introduced by Scholze \cite{Sch15}, we show that the action of the Hecke algebra on the space of overconvergent modular forms (with fixed radius) is locally analytic. As suggested by Matthew Emerton, this also reproves a result of Calegari-Emerton. At the end, we also discuss a similar phenomenon in the context of locally analytic vectors of completed cohomology.

\subsection*{Acknowledgement} 
I would like to thank Matthew Emerton for his comments on an earlier draft of this note and the  anonymous referee for many helpful comments on this paper.

\section{Locally analytic action}
\begin{defn}
Let $W$ be a $p$-adic Banach space over $\Q_p$. A continuous linear operator $T\in\End(W)$  is called \textit{locally analytic} if  there exists a monic polynomial $f(X)\in\Z_p[X]$ such that $f(T)(W^o)\subset pW^o$, where $W^o$ denotes the unit ball of $W$.
\end{defn}

Note that for a locally analytic operator $T$, if $W^o$ is $T$-stable, then the image of $T$ in $\End(W^o/pW^o)$ generates a finite $\F_p$-algebra.

\begin{exa}
Suppose $W$ is a finite dimensional vector space over $\Q_p$. Then any linear operator  of norm $\leq 1$ is locally analytic by considering its characteristic polynomial.
\end{exa}

\begin{exa}
Suppose $W=\Q_p\langle X\rangle$, the ($p$-adic) completion of $\Q_p[X]$ with respect to the unit ball $\Z_p[X]$. Let $T\in\End(\Q_p\langle X\rangle)$ be the translation $X\mapsto X+1$. It is locally analytic because $(T^p-1)\cdot F(X)=F(X+p)-F(X)\in p\Z_p[X]$ for any $F(X)\in\Z_p[X]$.
\end{exa}

Recall that an operator $T$ on $W$ is called topologically nilpotent if  $\displaystyle \lim_{n\to\infty}T^n\cdot v=0$ for any $v\in W$, i.e. the sequence $\{T^n\}_{n\geq 0}$ converges to zero in the space of linear operators on $W$ with respect to the \textit{weak topology}.

\begin{prop} \label{tnun}
Let $W$ be a $p$-adic Banach space over $\Q_p$. Suppose  $T\in\End(W^o)$ is topologically nilpotent.  The following are equivalent.
\begin{enumerate}
\item  $T$ is locally analytic;
\item  $T^n(W^o)\subseteq pW^o$ for some $n\geq 1$, i.e. $T^n\cdot v$ converges to $0$ uniformly for all $v\in W^o$;
\item The sequence $\{T^n\}_{n\geq 0}$ converges to zero in $\End(W^o)$ with respect to the $p$-adic topology (equivalently the norm topology).
\end{enumerate}
\end{prop}

\begin{proof}
(2) and (3) are clearly equivalently. (2) implies (1) by taking $f(T)=T^n$ in the definition of locally analytic operators. It remains to show that (1) implies (2). Suppose that $f(T)(W^o)\subset pW^o$ for some monic polynomial $f(X)\in\Z_p[X]$. Write $f(X)\mod p =X^k g(X)$ with $g(X)\in\F_p[X]$ and $g(0)\neq 0$. Then $g(T)$ is invertible on $W^o/pW^o$ as $T$ is topologically nilpotent. Hence $T^k=f(T)=0$ viewed as elements in $\End(W^o/p)$.
\end{proof}

We can also generalize this notion to representations of algebras.
\begin{defn}
Suppose $A$ is a ring and $W$ is a $p$-adic Banach space equipped with an $A$-module structure. We say the action of $A$ on $W$ is \textit{locally analytic} if there exists an $A$-stable open and bounded lattice $\mathcal{L}\subseteq W$ such that  the image of $A\to \End(\mathcal{L}/p^n\mathcal{L})$ is finite  for any $n\geq 1$. If this happens, 
 the image of  $A\to \End(\mathcal{L}'/p^n\mathcal{L}')$ is finite for any $n$ and any $A$-stable open and bounded lattice $\mathcal{L}'\subseteq W$.
\end{defn}

In some cases, we only need to consider the image of $A\to \End(\mathcal{L}/p\mathcal{L})$.
\begin{lem}\label{noela}
Suppose $A$ is a Noetherian ring and $W$ is a $p$-adic Banach space equipped with an $A$-module structure.  The action of $A$ on $W$ is locally analytic if there exists an $A$-stable open and bounded lattice $\mathcal{L}\subseteq W$ such that the image of $A\to \End(\mathcal{L}/p\mathcal{L})$ is finite.
\end{lem}

\begin{proof}
Let $I_n$ be the kernel of $A\to \End(\mathcal{L}/p^n\mathcal{L})$. Clearly $I_1^n\subseteq I_n$. Hence it is enough to show that $A/I_1^n$ is finite. The assumption implies that $A/I_1$ is finite. Since $A$ is Noetherian, we have $I_1^n/I_1^{n+1}$ is also finite. Our claim follows as $A/I_1^n$ is filtered by $I_1^k/I_1^{k+1}$, $k=0,\cdots,n-1$.
\end{proof}

The following Proposition explains our choice of the notion ``locally analytic".
\begin{prop} \label{lapLg}
Suppose $G$ is a compact $p$-adic Lie group and $W$ is a continuous $p$-adic Banach space representation of $G$. Then the following are equivalent.
\begin{enumerate}
\item there exists a $G$-stable open and bounded lattice $\mathcal{L}\subseteq W$ such that $\mathcal{L}/p\mathcal{L}$ is fixed by some open subgroup $G'$ of $G$;
\item the induced action of $\Z_p[G]$ (the group algebra of $G$) on $W$ is locally analytic;
\item the induced action of $\Z_p[[G]]$ (the Iwasawa algebra) on $W$ is locally analytic;
\item $W$ is an analytic representation of some open subgroup $G'$ of $G$. In particular, $W$ is a locally analytic representation of $G$ in the usual sense.
\end{enumerate}
\end{prop}

\begin{proof}
Note that  $\Z_p[[G]]$ is Noetherian, cf. \cite[Theorem 33.4, Theorem 27.1]{Sch11}. Hence by Lemma \ref{noela},
part (3) follows from (1) by noting that the action of $\Z_p[[G]]$ on $\mathcal{L}/p\mathcal{L}$ factors through $\F_p[G/G']$ for some open normal subgroup of $G$. Part (3) implies (2) because $\Z_p[G]\subseteq\Z_p[[G]]$. To see (2) implies (1), we may assume $G$ is pro-p by replacing it by an open subgroup.  Then $g-1$ is topologically nilpotent on $W$ for any $g\in G$. The same argument of Proposition \ref{tnun} shows that $(g-1)^{p^n}(\mathcal{L})\subseteq p\mathcal{L}$ for some $n>0$ and some open bounded $G$-stable lattice $\mathcal{L}$ and any $g\in G$. Hence $G^{p^n}$ fixes $\mathcal{L}/p\mathcal{L}$. Part (1) follows as $G^{p^n}$ contains an open subgroup of $G$, cf. Theorem 27.1 and Remark 26.9 of \cite{Sch11}. 

It remains to to prove the equivalence between (1) and (4). This is well-known. Part (4) follows from (1) by considering the Mahler coefficients and invoking Amice's theorem, cf. \cite[IV.4]{CD14}. Now assume (4). There is a $G'$-equivariant isomorphism $W\cong (W\widehat\otimes_{\Q_p} \sC^{an}(G',\Q_p))^{G'}$ where $G'$ acts on the right hand side via the right translation action on $\sC^{an}(G',\Q_p)$, the space of $\Q_p$-valued analytic functions on $G'$, cf. \cite[2.1]{Pan20}. Part (1) follows by noting that $\mathscr{C}^{\an}(G',\Q_p)^{o}/p$ is fixed by an open subgroup $G'$ by \cite[Lemma 2.1.2.]{Pan20}.
\end{proof}

\begin{exa} \label{kex}
Suppose $G=\Z_p^k$ and $W$ is a $\Q_p$-Banach space representation of $G$. Then $\Z_p[[G]]\cong \Z_p[[T_1,\ldots,T_k]]$ where $T_i=g_i-1$ and $g_1,\ldots, g_k$ form a basis of $G$. Now suppose the action of $\Z_p[[G]]$ on $W$ is locally analytic. It follows from the previous discussion that there exists $n>0$ such that $T_i^n(W^o)\subseteq pW^o$, $i=1,\ldots,k$, or equivalently $\frac{T_i^n}{p}$ has norm $\leq 1$.  Hence the action of $\Z_p[[T_1,\ldots,T_k]]$ on $W$ can be extended to $\Q_p\langle \frac{T_1^n}{p},\ldots, \frac{T_k^n}{p} ,T_1,\ldots,T_k\rangle$.

Geometrically, the generic fiber of $\Z_p[[T_1,\ldots,T_k]]$ is an open ball and the rigid analytic space associated to $\Q_p\langle \frac{T_1^n}{p},\ldots, \frac{T_k^n}{p} ,T_1,\ldots,T_k\rangle$ corresponds to a \textit{closed} polydisc inside. Roughly speaking, this means the spectrum of $W$ is in a bounded region with radius strictly less than $1$.
\end{exa}

\begin{rem}
Proposition \ref{lapLg} shows that one can extend the notion of locally analytic representations to general topological groups. More precisely, a $\Q_p$-Banach space representation $W$ of a topological group $G$ is called \textit{locally analytic} if the action of $\Z_p[G]$ on $W$ is locally analytic in our sense. When $G=G_{K}$, the local Galois group of a finite extension $K$ of $\Q_p$,  these locally analytic representations show up naturally in the recent development of Sen's theory. For example, one can show that for a locally analytic representation $W$ of $G_{K}$, there is a natural isomorphism
\[W\widehat\otimes_{\Q_p} \overbar{K}\cong (W\widehat\otimes_{\Q_p} \overbar{K})^{H_{K},\Gamma_{K,n}-\mathrm{an}}\widehat\otimes_{K_n} \overbar{K}\]
for some $n>0$. Here $H_K=\Gal(\overbar{K}/K(\mu_{p^\infty}))$, $K_n=K(\mu_{p^n})$, $\Gamma_{K,n}=\Gal(K(\mu_{p^\infty})/K_n)$, and the superscript $H_K$ denotes taking the $H_K$-invariants and ``$\Gamma_{K,n}-\mathrm{an}$'' denotes taking the $\Gamma_{K,n}$-analytic vectors. See \cite[Theorem 3.3.3]{CJER22} for a relative version of this result. We remark that when $G=G_K$, the equivalence between parts (1) and (3) of Lemma \ref{noela} still holds, even though $\Z_p[[G_K]]$ is not Noetherian. This is a consequence of the local class field theory: for any finite extension $L$ of $\Q_p$, the dimension of $\Hom(G_L,\F_p)$ is finite.

\end{rem}

\section{Fake-Hasse invariants} \label{FHi}
In order to study the Hecke action on overconvergent modular forms, we need fake-Hasse invariants and strange formal integral models of the modular curve constructed by Scholze in Chapter 4 of \cite{Sch15}. That the Hecke action is locally analytic will be a formal consequence of the existence of these Hecke-invariant sections.

Our setup is as follows. Let $C=\bC_p$ the $p$-adic completion of $\overbar\Q_p$ with ring of integers $\cO_C$. For a sufficiently small open compact subgroup $K$ of $\GL_2(\A_f)$, we denote by $X^*_{K,C}$ the complete adelic modular curve over $C$ of level $K$ and by $\mathcal{X}^*_{K}$ its associated rigid analytic space. We will always assume $K$ is sufficiently small so that $X^*_{K,C}$ is a variety. If we choose an isomorphism between $C$ and $\bC$, the non-cusp points of $X^*_{K,C}(C)$ are given by the usual double cosets $\GL_2(\Q)\setminus (\bC-\R)\times \GL_2(\A_f)/K$.
On $\mathcal{X}^*_{K}$, we have the usual automorphic line bundle $\omega_{K^pK_p}$. Fix an open compact subgroup $K^p\subseteq\GL_2(\A_f^p)$ contained in the level-$N$-congruence subgroup for some $N\geq 3$ prime to $p$. For a sufficiently small open subgroup $K_p\subseteq \GL_2(\Q_p)$, in the proof of Theorem 4.3.1. of \cite{Sch15}, Scholze constructed 
\begin{itemize}
\item a formal integral model $\mathfrak{X}^*_{K^pK_p}$ of $\mathcal{X}^*_{K^pK_p}$ together with an affine open cover $\mathfrak{V}_{K_p,1},\mathfrak{V}_{K_p,2}$;
\item an ample line bundle $\omega^{\mathrm{int}}_{K^pK_p}$ on $\mathfrak{X}^*_{K^pK_p}$ whose generic fiber is $\omega_{K^pK_p}$;
\end{itemize}
Moreover fix $n\geq 1$. For  a sufficiently small open subgroup $K_p\subseteq \GL_2(\Q_p)$, there are
\begin{itemize}
\item global sections $\bar{s}_{n,1},\bar{s}_{n,2}\in H^0(\mathfrak{X}^*_{K^pK_p},\omega^{\mathrm{int}}_{K^pK_p}/p^n)$ (fake-Hasse invariants) such that ${\mathfrak{V}_{K_p,i}}$ is the locus where $\bar{s}_{n,i}$ is invertible for $i=1,2$. In particular, $\bar{s}_{n,1},\bar{s}_{n,2}$ generate $\omega^{\mathrm{int}}_{K^pK_p}/p^n$. 
\end{itemize}
All $\mathfrak{X}^*_{K^pK_p},\mathfrak{V}_{K_p,1},\mathfrak{V}_{K_p,2}$ and $\omega^{\mathrm{int}}_{K^pK_p}$ are functorial in $K^pK_p$, hence $\GL_2(\A_f^p)$ acts on the tower of $(\mathfrak{X}^*_{K^pK_p},\omega^{\mathrm{int}}_{K^pK_p})$. Both sections $\bar{s}_{n,1},\bar{s}_{n,2}$ are invariant under this action. 

We briefly recall Scholze's construction. Scholze proved that when the level at $p$ varies, the inverse limit $\displaystyle \varprojlim_{K_p\subseteq \GL_2(\Q_p)}\mathcal{X}^*_{K^pK_p}$ exists as a perfectoid space, which will be denoted by $\mathcal{X}^*_{K^p}$. Moreover, there is the so-called Hodge-Tate period morphism 
\[\pi_{\HT}:\mathcal{X}^*_{K^p}\to \Fl\]
defined via the position of the Hodge-Tate filtration on the first cohomology of the universal elliptic curve (on the non-cusp points). Here $\Fl$ ($\cong\mathbb{P}^1$) denotes the associated adic space of the flag variety of $\GL_2/C$. The pull-back of the tautological ample line bundle $\omega_{\Fl}$ on $\Fl$ along $\pi_{\HT}$ is canonically identified with the pull-back of $\omega_{K^pK_p}$ to $\mathcal{X}^*_{K^p}$ (up to a Tate twist). Note that $\Gamma(\Fl,\omega_\Fl)$ has a canonical basis $f_1,f_2$, whose pull-back to $\mathcal{X}^*_{K^p}$ will be denoted by $e_1,e_2$. Let $U_1,U_2\subseteq \mathcal{X}^*_{K^p}$ be the open subsets defined by $\|e_2/e_1\|\leq 1$ and $\|e_1/e_2\|\leq 1$ respectively. Hence $e_i$ is an invertible section on $U_i$ for $i=1,2$. Scholze proved that $U_1$ and $U_2$ are affinoid perfectoid and are the preimages of some affinoid open subsets $V_{K_p,1},V_{K_p,2}$ of $\mathcal{X}^*_{K^pK_p}$ for sufficiently small $K_p$.  Fix $n\geq 1$. For a  sufficiently small subgroup $K_p$ and $i=1,2$, we may find 
\begin{itemize}
\item $s_{n,i}\in \Gamma(V_{K_p,i},\omega_{K^pK_p})$ such that $\displaystyle \|1-\frac{s_{n,i}}{e_i}\|\leq \|p^n\|$;
\item $x_{n,i}\in \Gamma(V_{K_p,i},\cO_{\mathcal{X}^*_{K^pK_p}})$ such that $\displaystyle \|\frac{e_{3-i}}{e_i}-x_{n,i}\|\leq \|p^n\|$.
\end{itemize}
This is possible because the natural map
\[\varinjlim_{K_p}\Gamma(V_{K_p,i},\cO_{\mathcal{X}^*_{K^pK_p}})\to \Gamma(U_{i},\cO_{\mathcal{X}^*_{K^p}})\]
has dense images. The formal model $\mathfrak{X}^*_{K^pK_p}$ is obtained by glueing $\mathfrak{V}_{K_p,1}:=\Spf \cO^+(V_{K_p,1})$ and $\mathfrak{V}_{K_p,2}:=\Spf \cO^+(V_{K_p,2})$ along $\Spf \cO^+(V_{K_p,1}\cap V_{K_p,2})$. The integral line bundle $\omega^{\mathrm{int}}_{K^pK_p}$ is defined by requiring $s_{n,i}$ being invertible on $\mathfrak{V}_i$ for $i=1,2$. This does not depend on $n$ and the choice of $s_{n,i}$. For the  fake-Hasse invariants, observe that $s_{n,1}\mod p^n$ and $s_{n,2}x_{n,2}\mod p^n$ glue a global section $\bar{s}_{n,1}\in H^0(\mathfrak{X}^*_{K^pK_p},\omega^{\mathrm{int}}_{K^pK_p}/p^n)$ by our choice of $s_{n,i},x_{n,i}$. Similarly one can construct $\bar{s}_{n,2}$. We remark that $\bar{s}_{n,1},\bar{s}_{n,2}$ are independent of the choice of $s_{n,i},x_{n,i}$ because $\bar{s}_{n,i}$ may be viewed as $e_i\mod p^n$. Thus $\bar{s}_{n,1},\bar{s}_{n,2}$ are fixed by the action of $\GL_2(\A^p_f)$.

Let $\T=\T_{K^p}=\Z_p[\GL_2(\A_f^p)//K^p]$ be the abstract Hecke algebra of $K^p$-biinvariant compactly supported functions on $\GL_2(\A_f^p)$, where the Haar measure gives $K^p$ measure $1$. Let $K_p$ be a sufficiently small subgroup of $\GL_2(\Q_p)$ so that $\mathfrak{X}^*_{K^pK_p}$ and $\mathfrak{V}_1,\mathfrak{V}_2$ are defined. (We drop some subscripts $K_p$ from the notations.) It follows from the functorial properties of $\mathfrak{V}_1,\mathfrak{V}_2$ that $H^0(\mathfrak{V}_i,(\omega^{\mathrm{int}}_{K^pK_p})^{\otimes k}),i=1,2$ and $k\in \Z$ admits a natural action of $\T$.  Denote by $V_i=V_{K_p,i}\subseteq\mathcal{X}^*_{K^pK_p}$ the generic fiber of $\mathfrak{V}_i$. Then $H^0(V_i,\omega_{K^pK_p}^{\otimes k})$ is a $p$-adic Banach space with unit ball $H^0(\mathfrak{V}_i,(\omega^{\mathrm{int}}_{K^pK_p})^{\otimes k})$. Our main result here is

\begin{thm} \label{mT}
For $i=1,2$ and $k\in\Z$, the Hecke action of $\T$ on $H^0(V_i,\omega_{K^pK_p}^{\otimes k})$ is locally analytic.
\end{thm}

\begin{rem}
We will relate $H^0(V_i,\omega_{K^pK_p}^{\otimes k})$ with classical overconvergent modular forms later in the next section. See the proof of Corollary \ref{HTS}.
\end{rem}

Since $\mathfrak{V}_i$ is affine, $H^0(\mathfrak{V}_i,(\omega^{\mathrm{int}}_{K^pK_p})^{\otimes k})/p^n=H^0(\mathfrak{V}_i,(\omega^{\mathrm{int}}_{K^pK_p})^{\otimes k}/p^n)$. It follows from the construction of $\omega^{\mathrm{int}}_{K^pK_p}$ that if $K'_p$ is an open subgroup of $K_p$, the pull-back map 
\[H^0(\mathfrak{V}_i,(\omega^{\mathrm{int}}_{K^pK_p})^{\otimes k}/p^n)\to H^0(\mathfrak{V}_{K_p',i},(\omega^{\mathrm{int}}_{K^pK'_p})^{\otimes k}/p^n)\]
is injective as $\bar{s}_{1,i}^k$ generates both $(\omega^{\mathrm{int}}_{K^pK_p})^{\otimes k}/p$ and $(\omega^{\mathrm{int}}_{K^pK'_p})^{\otimes k}/p$ on $\mathfrak{V}_i$ and $\mathfrak{V}_{K'_p,i}$ respectively. Hence for a fixed $n$, we are free to replace $K_p$ by a smaller subgroup. In particular, we may assume $\bar{s}_{n,i}$ exists. Note that $\bar{s}_{n,i}$ is an invertible section on $\mathfrak{V}_i$ and commutes with the Hecke actions. There are $\T$-equivariant isomorphisms:
\[H^0(\mathfrak{V}_i,\cO_{\mathfrak{X}^*_{K^pK_p}}/p^n)\stackrel{\times \bar{s}_{n,i}^{k}}{\to}H^0(\mathfrak{V}_i,(\omega^{\mathrm{int}}_{K^pK_p})^{\otimes k}/p^n).\]
Hence $H^0(\mathfrak{V}_i,(\omega^{\mathrm{int}}_{K^pK_p})^{\otimes k})/p^n$ is independent of $k$ as a Hecke module. Thus  it suffices to prove Theorem \ref{mT} for $k=0$ and we  have the following corollary.
\begin{cor} \label{cor1}
The Hecke actions of $\T$ on 
\[(\prod_{k\in\Z} H^0(\mathfrak{V}_i,(\omega^{\mathrm{int}}_{K^pK_p})^{\otimes k}))\otimes_{\Z_p}\Q_p,\,i=1,2\]
 are locally analytic.
\end{cor}

\begin{proof}[Proof of Theorem \ref{mT}]
Fix $n\geq 1$. By definition, we need to show that  the image of 
\[\T\to\End\left(H^0(\mathfrak{V}_i,\cO_{\mathfrak{X}^*_{K^pK_p}}/p^n)\right)\]
is finite. By shrinking $K_p$ if necessary, we may assume $\bar{s}_{n,i}$ exists.
Since ${\mathfrak{V}_i}$ is the locus where $\bar{s}_{n,i}$ is invertible, we may write 
\[H^0(\mathfrak{V}_i,\cO_{\mathfrak{X}^*_{K^pK_p}}/p^n)=\varinjlim_{\times\bar{s}_{n,i}}H^0(\mathfrak{X}^*_{K^pK_p},(\omega^{\mathrm{int}}_{K^pK_p})^{\otimes k}/p^n).\]
Hence it suffices to show the image of 
\[\T\to\End\left(H^0(\mathfrak{X}^*_{K^pK_p},(\omega^{\mathrm{int}}_{K^pK_p})^{\otimes k}/p^n)\right)\]
is finite and the kernel stabilizes when $k$ is sufficiently large. For the finiteness, by the ampleness of $\omega^{\mathrm{int}}_{K^pK_p}$, when  $k$ is sufficiently large, we have 
\[H^0\left(\mathfrak{X}^*_{K^pK_p},(\omega^{\mathrm{int}}_{K^pK_p})^{\otimes k}/p^n\right)=H^0\left(\mathfrak{X}^*_{K^pK_p},(\omega^{\mathrm{int}}_{K^pK_p})^{\otimes k}\right)/p^n.\]
Since $H^0\left(\mathfrak{X}^*_{K^pK_p},(\omega^{\mathrm{int}}_{K^pK_p})^{\otimes k}\right)$ is $p$-torsion free, it is enough to show that the image of 
\[\T\to\End_{\cO_C}\left(H^0(\mathfrak{X}^*_{K^pK_p},(\omega^{\mathrm{int}}_{K^pK_p})^{\otimes k})\right)\subseteq\End_C\left(H^0(\mathcal{X}^*_{K^pK_p},\omega_{K^pK_p}^{\otimes k})\right)\]
is a finite $\Z_p$-module. Indeed, the properness of $\mathfrak{X}^*_{K^pK_p}$ implies that $H^0(\mathfrak{X}^*_{K^pK_p},(\omega^{\mathrm{int}}_{K^pK_p})^{\otimes k})$ is a finite $\cO_C$-module and $H^0(\mathcal{X}^*_{K^pK_p},\omega_{K^pK_p}^{\otimes k})$ is a finite dimensional $C$-vector space. Our claim is clear as $\mathcal{X}^*_{K^pK_p}$, the sheaf $\omega_{K^pK_p}$ and Hecke actions are all defined over $\Q_p$.

To see that the kernel of $\T\to\End\left(H^0(\mathfrak{X}^*_{K^pK_p},(\omega^{\mathrm{int}}_{K^pK_p})^{\otimes k}/p^n)\right)$ stabilizes, consider the exact sequence
\[0\to (\omega^{\mathrm{int}}_{K^pK_p})^{\otimes k-1}/p^n\stackrel{(\bar{s}_{n,1},\bar{s}_{n,2})}{\longrightarrow} (\omega^{\mathrm{int}}_{K^pK_p})^{\otimes k}/p^n\oplus (\omega^{\mathrm{int}}_{K^pK_p})^{\otimes k}/p^n
\stackrel{(\bar{s}_{n,2},-\bar{s}_{n,1})}{\longrightarrow}
(\omega^{\mathrm{int}}_{K^pK_p})^{\otimes k+1}/p^n\to 0.\]
(This essentially comes from the non-split sequence $0\to \cO(-1)\to\cO^{\oplus 2}\to\cO(1)\to 0$ on $\mathbb{P}^1$.)
When $k$ is sufficiently large, taking global sections of this exact sequence remains exact as $\omega^{\mathrm{int}}_{K^pK_p}$ is ample. Thus the Hecke action of $\T$ on $H^0(\mathfrak{X}^*_{K^pK_p},(\omega^{\mathrm{int}}_{K^pK_p})^{\otimes k+1}/p^n)$ factors through $H^0(\mathfrak{X}^*_{K^pK_p},(\omega^{\mathrm{int}}_{K^pK_p})^{\otimes k}/p^n)^{\oplus 2}$, which proves the claim.
\end{proof}

\section{Hodge-Tate-Sen weights}
In this section, we study Galois representations attached to eigenforms in $H^0(V_i,\omega_{K^pK_p}^{\otimes k})$.  Let me introduce some (standard) notation first. For simplicity, from now on we assume $K^p\subset\GL_2(\A_f^p)$ is of the form $\prod_{l\neq p}K_l$. Let $S$ be a finite set of rational primes containing $p$ such that $K_l\cong \GL_2(\Z_l)$ for $l\notin S$. Denote by $\T_S=\Z_p[\GL_2(\A_f^S)//\prod_{l\notin S}K_l]\subset \T$ the subalgebra generated by spherical Hecke operators. Consider the image $\T_{i,1}$ of $\T_S\to \End(H^0(\mathfrak{V}_i,\cO_{\mathfrak{X}^*_{K^pK_p}}/p))$. By Theorem \ref{mT}, this is a finite $\F_p$-algebra. Moreover, by Corollary 5.11 of \cite{Sch15},  there is a continuous $2$-dimensional determinant $D$ of $G_{\Q,S}$ valued in $\T_{i,1}$ in the sense of Chenevier \cite{Che14} satisfying the following property: for any $l\notin S$, the characteristic polynomial of $D(\Frob_l)$ is
\[X^2-l^{-1}T_lX+l^{-1}S_l.\]
Here $G_{\Q,S}$ denotes the Galois group of the maximal extension of $\Q$ unramified outside of $S$ and infinity, $\Frob_l\in G_{\Q,S}$ denotes a geometric Frobenius element at $l$ and $T_l,S_l$ denote the usual Hecke operators
\[[K_l\begin{pmatrix}l&0\\0&1\end{pmatrix}K_l],\,[K_l\begin{pmatrix}l&0\\0&l\end{pmatrix}K_l].\]
Let $\F$ be a finite field so that  all residue fields of $\T_{i,1}$ can be embedded into $\F$. Fix an embedding of $W(\F)[\frac{1}{p}]$ into $\overbar\Q_p$, or equivalently an embedding $\F\to \cO_C/p$. Then $\T_{i,1}\otimes_{\F_p}\F$ acts on $H^0(\mathfrak{V}_i,\cO_{\mathfrak{X}^*_{K^pK_p}}/p)$ and we denote by $\T_{i}$ its image in $ \End(H^0(\mathfrak{V}_i,\cO_{\mathfrak{X}^*_{K^pK_p}}/p))$. Finally, for any maximal ideal $\km$ of $\T_i$, we have a continuous $2$-dimensional determinant $D_{\km}$ of $G_{\Q,S}$ valued in $\T_{i}/\km=\F$. Let $R^{\ps}_{\km}$ be the universal formal $W(\F)$-algebra parametrizing all liftings of $D_{\km}$. This is a noetherian ring. Denote the product over all $\km$ by
\[R^{\ps}=\prod_{\km\in\Spec\T_i}R^{\ps}_{\km}.\]
Now for any $k\in\Z,n>0$, by Corollary 5.1.11 of \cite{Sch15}, there is a lifting of $\prod_{\km\in\Spec\T_i}D_{\km}$ valued in the image of $\T_S\otimes_{\Z_p}W(\F)\to\End(H^0(\mathfrak{X}^*_{K^pK_p},(\omega^{\mathrm{int}}_{K^pK_p})^{\otimes k}/p^n))$. By the universal property, this image receives a map from $R^{\ps}$. Hence we obtain an action of  $R^{\ps}$ on $H^0(V_i,\omega_{K^pK_p}^{\otimes k})$ factoring through the Hecke action. In particular, by Corollary \ref{cor1},

\begin{cor}
The action of $R^{\ps}$ on $(\prod_{k\in\Z} H^0(\mathfrak{V}_i,(\omega^{\mathrm{int}}_{K^pK_p})^{\otimes k}))\otimes_{\Z_p}\Q_p$
 is locally analytic.
\end{cor}
Concretely, since each $R^{\ps}_{\km}$ is a noetherian local formal $W(\F)$-algebra, it can be written as a quotient of $W(\F)[[x_1,\ldots,x_g]]$ for some $g$. As explained in Example \ref{kex}, there exists an integer $n>0$ such that $\frac{x_j^n}{p}$ has norm $\leq 1$ acting on $(\prod_{k\in\Z} H^0(\mathfrak{V}_i,(\omega^{\mathrm{int}}_{K^pK_p})^{\otimes k}))\otimes_{\Z_p}\Q_p$ for any $j=1,\ldots,g$. Therefore, let $E\subset\overbar\Q_p$ be a finite extension of $W(\F)[\frac{1}{p}]$ containing a $n$-th root of $p$ and fix such a root $p^{1/n}\in E$. We can extend the action of $W(\F)[[x_1,\ldots,x_g]]$ to an $E$-linear action of $E\langle\frac{x_1}{p^{1/n}},\ldots,\frac{x_g}{p^{1/n}}\rangle$ on $(\prod_{k\in\Z} H^0(\mathfrak{V}_i,(\omega^{\mathrm{int}}_{K^pK_p})^{\otimes k}))\otimes_{\Z_p}\Q_p$. Recall that geometrically, the generic fiber of $W(\F)[[x_1,\ldots,x_g]]$ is an open ball and $E\langle\frac{x_1}{p^{1/n}},\ldots, \frac{x_g}{p^{1/n}}\rangle$ corresponds to a \textit{closed} polydisc inside. This means the spectrum of $H^0(V_i,\omega_{K^pK_p}^{\otimes k})$ is in a bounded region with radius strictly less than $1$.

We make such a choice for each $\km$. As a consequence,  the action of $R^{\ps}$ can be extended to an action of a topologically finitely generated Banach $E$-algebra. We denote its image in 
\[\End((\prod_{k\in\Z} H^0(\mathfrak{V}_i,(\omega^{\mathrm{int}}_{K^pK_p})^{\otimes k}))\otimes_{\Z_p}\Q_p)\]
by $\mathcal{R}$. There is a natural map $R^{\ps}\to\mathcal{R}$. Hence we have a $2$-dimensional determinant $D_{\mathcal{R}}$ of $G_{\Q,S}$ valued in $\mathcal{R}$ which is continuous with respect to the $p$-adic topology on $\mathcal{R}$. The whole point of showing  that the  Hecke action is locally analytic is to improve the continuity of the determinant on $R^{\ps}$ from the $\mathrm{rad}(R^{\ps})$-adic  topology to a $p$-adic topology.

The Hecke action on $H^0(\mathcal{X}^*_{K^pK_p},\omega_{K^pK_p}^{\otimes k}),k\geq 0$ extends naturally to $\mathcal{R}$. In fact, the image of $\T_S\otimes E$ in $\End(H^0(\mathcal{X}^*_{K^pK_p},\omega_{K^pK_p}^{\otimes k}))$ agrees with the image of  $\mathcal{R}$. In particular, the action of $\mathcal{R}$ on $H^0(\mathcal{X}^*_{K^pK_p},\omega_{K^pK_p}^{\otimes k})$ is semi-simple. 

\begin{lem} \label{dense}
The kernel of 
\[\mathcal{R}\to \End(\prod_{k\geq 0}H^0(\mathcal{X}^*_{K^pK_p},\omega_{K^pK_p}^{\otimes k}))\]
is trivial.
\end{lem}

\begin{proof}
This is a standard application of fake Hasse invariants. See the proof of Theorem 4.4.1. of \cite{Sch15}. We give a sketch here. Suppose $f\in\mathcal{R}$ is a non-zero element in the kernel of the above map. We may assume it has norm $\leq 1$ acting on $\prod_{k\in\Z} H^0(\mathfrak{V}_i,(\omega^{\mathrm{int}}_{K^pK_p})^{\otimes k})$ and its image in $\End (\prod_{k\in\Z} H^0(\mathfrak{V}_i,(\omega^{\mathrm{int}}_{K^pK_p})^{\otimes k}/p))$ is non-zero. Now since $R^{\ps}\otimes_{W(\F)} E\to\mathcal{R}$ has dense image, the action of $f$ on $H^0(\mathfrak{V}_i,(\omega^{\mathrm{int}}_{K^pK_p})^{\otimes k}/p)$ commutes with $\bar{s}_{1,i}^{n}$ if $n$ is sufficiently divisible by $p$. Indeed,  since $\omega^{\mathrm{int}}_{K^pK_p}$ is ample, $\bar{s}_{1,i}^{l}$ lifts to a global section $s_1\in H^0(\mathfrak{X}^*_{K^pK_p},(\omega^{\mathrm{int}}_{K^pK_p})^{\otimes l})$ for $l$ large enough. Hence $(g-1)\cdot s_1\in pH^0(\mathfrak{X}^*_{K^pK_p},(\omega^{\mathrm{int}}_{K^pK_p})^{\otimes l})$ for any $g\in\GL_2(\A_f^S)$. Thus $(g-1)\cdot s_1^{p^k} \in p^{k+1}H^0(\mathfrak{X}^*_{K^pK_p},(\omega^{\mathrm{int}}_{K^pK_p})^{\otimes lp^k})$ for $k\geq0$. In particular, 
\[T(s_1^{p^k}x)-s_1^{p^k} T(x)\in p^{k+1} H^0(\mathfrak{V}_i,(\omega^{\mathrm{int}}_{K^pK_p})^{\otimes m+lp^k})\]
 for $T\in\T_S$ and $x\in H^0(\mathfrak{V}_i,(\omega^{\mathrm{int}}_{K^pK_p})^{\otimes m})$. By continuity, this also holds for $T\in R^{\ps}$. If we write $f=\frac{f'}{p^k}+pf''$, where $f'\in R^{\ps}\otimes \cO_E$ and $f''\in \mathcal{R}$ with norm  $\leq 1$ acting on $\prod_{k\in\Z} H^0(\mathfrak{V}_i,(\omega^{\mathrm{int}}_{K^pK_p})^{\otimes k})$. It follows that $\frac{f'}{p^k}$ and $f$ commute with $(\bar{s}_{1,i})^{lp^k}$. This means $f$ acts non-trivially on $H^0(\mathfrak{X}^*_{K^pK_p},(\omega^{\mathrm{int}}_{K^pK_p})^{\otimes k}/p)$ for some sufficiently large $k$ by the same argument as in the proof of Corollary \ref{cor1}. In this case, $H^0(\mathfrak{X}^*_{K^pK_p},(\omega^{\mathrm{int}}_{K^pK_p})^{\otimes k}/p)=H^0(\mathfrak{X}^*_{K^pK_p},(\omega^{\mathrm{int}}_{K^pK_p})^{\otimes k})/p$ by the ampleness of $\omega^{\mathrm{int}}_{K^pK_p}$. But this contradicts our assumption on $f$.
\end{proof}

Recall that there is a determinant $D_{\mathcal{R}}$ of $G_{\Q,S}$ valued in $\mathcal{R}$. Since $\mathcal{R}$ is over a characteristic zero field, one can also view this as a function $T:G_{\Q,S}\to\mathcal{R}$, which behaves like the trace of a two-dimensional representation, i.e. a pseudo-representation. For any non-zero $E$-algebra homomorphism $\lambda:\mathcal{R}\to\overbar\Q_p$, we can associate a two-dimensional semi-simple continuous representation $\rho_\lambda:G_{\Q,S}\to\GL_2(\overbar\Q_p)$, well-defined up to conjugation, whose trace is given by $\lambda\circ T$.  Moreover, if $\lambda$ arises from an eigenform in $H^0(\mathcal{X}^*_{K^pK_p},\omega_{K^pK_p}^{\otimes k})$, then by Faltings's result \cite{Fa87}, $\rho_\lambda|_{G_{\Q_p}}$ has Hodge-Tate weights $0,k-1$. Our convention is that the cyclotomic character has Hodge-Tate weight $-1$. The density result \ref{dense} has the following consequence.

\begin{thm} \label{mtS}
For any $\lambda:\mathcal{R}\to\overbar\Q_p$, one of the Hodge-Tate-Sen weights of $\rho_\lambda|_{G_{\Q_p}}$ is $0$, i.e. $(\rho_\lambda\otimes_{\Q_p}C)^{G_{\Q_p}}\neq 0$.
\end{thm}

\begin{proof}
Recall that given a continuous representation of $G_{\Q_p}\to\GL_n(\overbar\Q_p)$, Sen constructs a monic polynomial $P_{\mathrm{Sen},\rho}$ of degree $n$ with coefficients in $\overbar\Q_p\otimes_{\Q_p}\Q_p(\mu_{p^\infty})$. It is called the Sen polynomial of $\rho$ and only depends on the semi-simplification of $\rho$. Its roots are called the Hodge-Tate-Sen weights of $\rho$ (or up to a sign depending on the normalization). Moreover, Sen shows that this polynomial varies analytically in family. See \cite{Sen88,Sen93} and also Th\'eor\`eme 5.1.4. of \cite{BC08}. We are going to apply Sen's theory in our context. 

First, suppose that there exists a continuous Galois representation $\rho_{\mathcal{R}}:G_{\Q,S}\to\GL_2(\mathcal{R})$ whose trace is $T$. Then by Sen's result, we can find a polynomial $P_{\mathrm{Sen},\rho_{\mathcal{R}}}$ with coefficients in $\mathcal{R}\otimes_{\Q_p}\Q_p(\mu_{p^\infty})$, such that for any $\lambda:\mathcal{R}\to\overbar\Q_p$, the Sen polynomial of $\rho_{\lambda}$ is given by $\lambda( P_{\mathrm{Sen},\rho_{\mathcal{R}}})$. By Lemma \ref{dense} and Faltings's result, the constant term of $P_{\mathrm{Sen},\rho_{\mathcal{R}}}$ vanishes as it vanishes after composing with any $\lambda$ arisen from an eigenform  in $H^0(\mathcal{X}^*_{K^pK_p},\omega_{K^pK_p}^{\otimes k})$. (Implicitly we using $\Q_p(\mu_p^\infty)$ is flat over $\Q_p$.) This immediately implies our claim.

In general, we may assume $\mathcal{R}$ is an integral domain. We are going to use the following lemma, whose proof will be given later.
\begin{lem} \label{uSen}
Assume $\mathcal{R}$ is normal. There exists a polynomial $P\in \mathcal{R}\otimes_{\Q_p}\Q_p(\mu_{p^\infty})[X]$  such that  for any $\lambda:\mathcal{R}\to\overbar\Q_p$, the Sen polynomial of $\rho_{\lambda}$ is $\lambda( P)$.
\end{lem}

Let $\mathcal{R}'$ be the normal closure of $\mathcal{R}$ in its fraction field. Note that $\mathcal{R}$ is a quotient of products of $E\langle x_1,\ldots,x_k\rangle$. Hence it is excellent because the Tate algebra $E\langle x_1,\ldots,x_k\rangle$ is excellent by the weak Jacobian condition (\cite[Theorem102]{Mat2}). In particular, $\mathcal{R}$ is a Nagata ring and $\mathcal{R}'$ is a finite $\mathcal{R}$-algebra. Thus $\mathcal{R}'$ is a Banach $E$-algebra. 

Now consider the pseudo-representation $G_{\Q,S}\stackrel{T}{\to} \mathcal{R}\to\mathcal{R}'$. Note that by the going-up property of integral extension, any $\lambda:\mathcal{R}\to\overbar\Q_p$ can be extended to a map $\lambda':\mathcal{R}'\to\overbar\Q_p$ and $\rho_{\lambda}\cong\rho_{\lambda'}$. In particular, it is enough to show that $\rho_{\lambda'}$ has a Hodge-Tate-Sen weight zero for any $\lambda':\mathcal{R}'\to\overbar\Q_p$. Applying the previous lemma to $\mathcal{R}'$, we get a universal Sen polynomial $P$ with coefficients in $\mathcal{R}'\otimes_{\Q_p}\Q_p(\mu_{p^\infty})$. Again it suffices to show the constant term of $P$ vanishes. Write the constant term of $P$ as $\sum_{i=1}^l a_i\otimes b_i$ with $a_i\in\mathcal{R}',b_i\in \Q_p(\mu_p^\infty)$ and $b_i$ are linearly independent over $\Q_p$. If one of $a_i$ is non-zero, say $a_1$, we can find a monic polynomial $Q(X)\in \mathcal{R}[X]$ with constant term $Q(0)\neq 0$ and $Q(a_1)=0$. By Lemma \ref{dense}, there exists a $\lambda:\mathcal{R}\to\overbar\Q_p$ arisen from an eigenform in $H^0(\mathcal{X}^*_{K^pK_p},\omega_{K^pK_p}^{\otimes k})$ and $\lambda(Q(0))\neq 0$. Let $\lambda':\mathcal{R}'\to\overbar\Q_p$ be a map extending $\lambda$. By Faltings's result, $\lambda'(a_1)=0$. But $0=\lambda'(Q(a_1))=\lambda(Q(0))\neq 0$. Contradiction. Thus we prove $P(0)=0$.
\end{proof}

\begin{proof}[Proof of Lemma \ref{uSen}]
First let me recall some standard constructions in the theory of pseudo-representations. Fix a complex conjugation $\sigma^*\in G_{\Q,S}$. Our pseudo-representation $T$ is odd in the sense that $T(\sigma^*)=0$. For any $\sigma,\tau\in G_{\Q,S}$, let
\begin{itemize}
\item $a(\sigma)=\frac{T(\sigma^*\sigma)+T(\sigma)}{2}$;
\item $d(\sigma)=T(\sigma)-a(\sigma)$;
\item $x(\sigma,\tau)=a(\sigma\tau)-a(\sigma)a(\tau)$.
\end{itemize}
We denote by $\mathcal{I}$ the ideal of $\mathcal{R}$ generated by all $x(\sigma,\tau)$. It is called the ideal of reducibility as $\rho_\lambda$ is reducible if and only if $\lambda(\mathcal{I})=0$. If $\mathcal{I}$ is generated by  some $x(\sigma_0,\tau_0)\neq 0$, then
\[\sigma\in G_{\Q,S}\mapsto \begin{pmatrix} a(\sigma) & \frac{x(\sigma,\tau_0)}{x(\sigma_0,\tau_0)}\\ x(\sigma_0,\sigma) & d(\sigma)\end{pmatrix}\]
defines a representation $G_{\Q,S}\to\GL_2(\mathcal{R})$ whose trace is $T$. In this case, our claim follows from Sen's result directly.

In general, $\mathcal{I}$ might not even be principal. Here is a sketch of our strategy. $\mathcal{X}:=\Spm \mathcal{R}$ is viewed as an affinoid rigid analytic variety. Consider the blowup $\tilde{\mathcal{X}}$ of $\mathcal{X}$ along the ideal sheaf defined by $\mathcal{I}$. Then $\mathcal{I}$ becomes an invertible sheaf on $\tilde{\mathcal{X}}$ and we can apply the previous construction and glue a polynomial on $\tilde{\mathcal{X}}$  interpolating  Sen polynomial at each point. Now the normal assumption guarantees that the coefficients of this polynomial actually belong to $\mathcal{R}$. This gives the polynomial we are looking for. Since everything is relatively simple here, the blowup process will be replaced by the explicit construction below. But it seems helpful to keep this blowup picture in mind.

If $\mathcal{I}= 0$, then $a,d$ are characters and our claim is clear. So we may assume $\mathcal{I}\neq 0$ from now on. Let $x_1=x(\sigma_1,\tau_1),\ldots, x_r=x(\sigma_r,\tau_r)$ be a set of non-zero generators of $\mathcal{I}$. Denote by $\mathcal{R}^+$ the unit ball of $\mathcal{R}$ and by $\mathcal{K}$ the fraction field of $\mathcal{R}$. For each $i\in\{1,\ldots,r\}$, we define $\mathcal{R}_i^+$ as  the $p$-adic completion of $\mathcal{R}^+[\frac{x_1}{x_i},\ldots,\frac{x_r}{x_i}]\subset \mathcal{K}$ and $\mathcal{R}_i=\mathcal{R}_i^+[\frac{1}{p}]$. Consider the pseudo-representation $G_{\Q,S}\stackrel{T}{\to} \mathcal{R}\to\mathcal{R}_i$. The ideal of reducibility in this case is generated by $x_i$. Hence we have a polynomial $P_i\in\mathcal{R}_i\otimes_{\Q_p}\Q_p(\mu_{p^\infty})[X]$ interpolating Sen polynomial at each point of $\Spm \mathcal{R}_i$. 

Denote by $\mathcal{Y}_i\subset\Spm\mathcal{R}_i$ the open subset defined by $x_i\neq 0$ and by $\mathcal{X}_i\subset\mathcal{X}$ the open subset defined by $x_i\neq 0,\|x_j\|\leq \|x_i\|,j=1,\ldots,r$. It is easy to see that $\mathcal{Y}_i$ maps isomorphically onto $\mathcal{X}_i$ under the natural map $\pi_i:\Spm \mathcal{R}_i\to \mathcal{X}$. Hence we may view $P_i|_{\mathcal{Y}_i}$ as a polynomial on $\mathcal{X}_i$. Clearly, it interpolates Sen polynomial at each point in $\mathcal{X}_i$. Hence we can glue all $P_i$ and get a polynomial $P$ on $\mathcal{X}':=\mathcal{X}\setminus V(\mathcal{I})$, the locus of irreducible representations. (Here we are using $\mathcal{R}$ is reduced.) Since $\mathcal{R}$ is normal and the coefficients of $P$ are bounded functions, by Bartenwerfer's result \cite[\S 3]{Bar76}, the coefficients of $P$ can be extended to functions defined everywhere on $\mathcal{X}$, i.e. $P\in\mathcal{R}\otimes_{\Q_p}\Q_p(\mu_{p^\infty})[X]$.

We claim this polynomial $P$ interpolates the Sen polynomial at each point in $\mathcal{X}$. By construction, this is true for points in $\mathcal{X}'$. It remains to verify points in $V(\mathcal{I})$. Let $\lambda:\mathcal{R}\to \overbar\Q_p$ be a non-zero map whose kernel contains $\mathcal{I}$. Note that there exists $i\in\{1,\ldots,r\}$ such that $\lambda$ can be extended to a map $\lambda':\mathcal{R}[\frac{x_1}{x_i},\ldots,\frac{x_r}{x_i}]\to \overbar\Q_p$. This is because the usual blowup (in algebraic geometry) of $\Spec \mathcal{R}$ along $\mathcal{I}$ maps surjectively onto $\Spec\mathcal{R}$. Fix an integer $n$ so that $\lambda'(\mathcal{R}^+[\frac{p^nx_1}{x_i},\ldots,\frac{p^nx_r}{x_i}])$ is contained in the ring of integers of $\overbar\Q_p$. We define $\mathcal{R}'^+_i$ as the $p$-adic completion of $\mathcal{R}^+[\frac{p^nx_1}{x_i},\ldots,\frac{p^nx_r}{x_i}]$ and $\mathcal{R}'_i=\mathcal{R}'^+_i[\frac{1}{p}]$. Then $\lambda'$ extends to $\mathcal{R}'_i$ naturally. Again there is a polynomial $P_i'\in\mathcal{R}'_i\otimes_{\Q_p}\Q_p(\mu_{p^\infty})[X]$ interpolating Sen polynomial at each point of $\Spm \mathcal{R}'_i$. It suffices to prove that $P$, considered as an element of $(\mathcal{R}'_i)^{\mathrm{red}}\otimes_{\Q_p}\Q_p(\mu_{p^\infty})[X]$, agrees with $P'_i$. Clearly this is true for points in the irreducible locus $\Spm\mathcal{R}'_i\setminus V(x_i)$. But this also implies points in $V(x_i)$ as $x_i$ is not a zero-divisor in $\mathcal{R}'_i$ by the flatness of $\mathcal{R}'^+_i$ over $\mathcal{R}^+[\frac{p^nx_1}{x_i},\ldots,\frac{p^nx_r}{x_i}]\subset\mathcal{K}$. (Note that $\mathcal{R}'^+_i$ is the $p$-adic completion of the noetherian ring $\mathcal{R}^+[\frac{p^nx_1}{x_i},\ldots,\frac{p^nx_r}{x_i}]$.) This finishes the proof.
\end{proof}

\begin{rem}
In fact, the normal assumption in Lemma \ref{uSen} can be waived here because the local universal deformation ring at $p$ of a pseudo-representation is normal by \cite[Theorem 1.4]{BIP21}.
\end{rem}

\begin{cor} \label{HTS}
The two dimensional semi-simple Galois representation associated to an overconvergent eigenform of weight $k\in\Z$ has Hodge-Tate-Sen weights $0,k-1$.
\end{cor}

\begin{proof}
We use the (generalized) notion of overconvergent modular forms introduced in \cite[Definition 5.2.5]{Pan20}. See also the  discussion there for its relation with classical overconvergent modular forms. For an open compact subgroup $K_p$ of $\GL_2(\Q_p)$, denote by $\mathcal{V}_{K_p}$ the set of open subsets $V\subseteq \mathcal{X}^*_{K^pK_p}$ such that 
\begin{itemize}
\item $\pi_{K_p}^{-1}(V)=\pi_{\HT}^{-1}(V_\infty)$ for some open neighborhood $V_\infty$ of $\infty\in\mathbb{P}^1=\Fl$.
\end{itemize}
Here $\pi_{K_p}:\mathcal{X}^*_{K^p}\to \mathcal{X}^*_{K^pK_p}$ denotes the projection map and $\pi_{\HT}:\mathcal{X}^*_{K^p}\to \Fl$ is the Hodge-Tate period morphism, cf. the discussion in the beginning of Section \ref{FHi}. For example $V_{K_p,2}$ introduced in Section \ref{FHi} is an element of $\mathcal{V}_{K_p}$ if $K_p$ is sufficiently small. Open sets in $\mathcal{V}_{K_p}$ form a projective system by inclusions. If $K'_p\subseteq K_p$ is an open subgroup, there is a natural map $\mathcal{V}_{K_p}\to \mathcal{V}_{K'_p}$ induced by taking the preimages. The space of overconvergent modular forms of weight $k$  is defined as 
\[M^{\dagger}_k(K^p):=\varinjlim_{K_p\to 1}\varinjlim_{V\in\mathcal{V}_{K_p}} H^0(V,\omega^k_{K^pK_p}).\]
(This is equivalent with \cite[Definition 5.2.5]{Pan20} by \cite[Proposition 5.2.6, Lemma 5.2.9]{Pan20}.) 
Fix a $V\in \mathcal{V}_{K_p}$. The Hecke operators away from $p$ acts on $H^0(V,\omega^k_{K^pK_p})$. An (non-zero) eigenvector of $\T_S$ in $M^{\dagger}_k(K^p)$ is called an overconvergent eigenform of weight $k$. We remark that elements of form $\begin{pmatrix} p^l & 0\\ 0 & 1\end{pmatrix}\in\GL_2(\Q_p)$ act on $M^{\dagger}_k(K^p)$.

Let $M_2\subseteq M^\dagger_k(K^p)$ be the image of $ \varinjlim_{K_p\to 1} H^0(V_{K_p,2},\omega^k_{K^pK_p})\to M^\dagger_k(K^p)$. We claim that
\[M^{\dagger}_k(K^p)=\bigcup_{n\in\Z} \begin{pmatrix} p^n& 0\\ 0 & 1\end{pmatrix}\cdot M_2.\]
This implies the corollary. Indeed, the action of $\begin{pmatrix} p^n& 0\\ 0 & 1\end{pmatrix}$ commutes with the action of the Hecke algebra. Hence our assertion follows from Theorem \ref{mtS} because the (usual) determinant of the Galois representation associated to an overconvergent eigenform of weight $k$ has Hodge-Tate weight $k-1$. To prove the claim, we first note that $\pi_{K_p}^{-1}(V_{K_p,2})=\pi_{\HT}^{-1}(U)$ for some open subset $U$ of $\Fl$ containing $\infty$, which is independent of $K_p$. For any open neighborhood $V_\infty$ of $\infty$,  $ \begin{pmatrix} p^n& 0\\ 0 & 1\end{pmatrix}\cdot U\subseteq V_\infty$ for some $n$. This implies that given $V\in\mathcal{V}_{K_p}$, there exist a sufficiently small open subgroup $K'\subseteq \GL_2(\Q_p)$ and integer $n$ such that 
\begin{itemize}
\item $gK'g^{-1}\subseteq K_p$, where $g= \begin{pmatrix} p^n& 0\\ 0 & 1\end{pmatrix}$;
\item under the isomorphism $\varphi:\mathcal{X}^*_{K^pK'}\stackrel{\sim}{\to}\mathcal{X}^*_{K^pgK'g^{-1}}$ induced by $g^{-1}$, we have $\varphi(V_{K',2})\subseteq \pi^{-1}(V)$ where $\pi:\mathcal{X}^*_{K^pgK'g^{-1}}\to \mathcal{X}^*_{K^pK_p}$ denotes the projection map.
 \end{itemize}
Thus the map $H^0(V,\omega_{K^pK_p}^k)\to M^\dagger_k(K^p)$ factors through $g\cdot H^0(V_{K',2},\omega_{K^pK'}^k)$ and our claim follows.
\end{proof}

\section{A result of Calegari-Emerton}
Matthew Emerton pointed out the following consequence of Corollary \ref{cor1}, which reproves a result of Calegari-Emerton \cite[Theorem 2.2]{CE04} and can be viewed as some evidence towards a question of Buzzard \cite[Question 4.4]{Bu05} asking whether for a fixed level, all Hecke eigenvalues of arbitrary weights lie in a \textit{finite} extension of $\Q_p$. We denote by $\overbar\Z_p$ the ring of integers of $\overbar\Q_p$ and by $\mathfrak{m}$ its maximal ideal.

\begin{thm}
Let $S$ be a finite set of rational primes containing $p$ and $K=\prod_{l}K_l$ be an open compact subgroup of $\GL_2(\A_f)$ with $K_l\cong\GL_2(\Z_l)$ for $l\notin S$. There exists a rational number $\kappa=\kappa(K,p)$ such that for any $\lambda:\T_S\to\overbar\Z_p$ appearing in $H^0(V_i,\omega_{K^pK_p}^{\otimes k})$ and $\lambda':\T_S\to\overbar\Z_p$ appearing in $H^0(V_i,\omega_{K^pK_p}^{\otimes k'})$ for some integers $k,k'$, ($\lambda,\lambda'$ may come from classical forms for example,) if $\lambda\equiv \lambda' \mod \mathfrak{m}$, then 
\[\lambda\equiv \lambda' \mod p^{\kappa}\overbar\Z_p.\]
\end{thm}

\begin{proof}
Clear as the action of $\T_S$ on $H^0(V_i,\omega_{K^pK_p}^{\otimes k})$ is locally analytic.
\end{proof}

\section{Hecke action on locally analytic vectors of admissible representations}
In this last section, we provide another example of locally analytic Hecke actions: the case of the Hecke algebra acting on the locally analytic vectors in the completed cohomology. In fact we will prove this result in a more general setup. Suppose 
\begin{itemize}
\item $G$ is a finite-dimensional $p$-adic Lie group;
\item  $W$ is an \textit{admissible} Banach space representation of $G$. Recall that this means that for any open compact subgroup $K$ of $G$ and any open bounded $K$-stable lattice $\mathcal{L}\subseteq W$, the $\F_p$-dimension of $(\mathcal{L}/p\mathcal{L})^{K}$ is finite.
\item $A$ is a ring  and $W$ is equipped with an $A$-module structure which commutes with $G$.
\end{itemize}
For simplicity, we also assume the following:
\begin{itemize}
\item $W^o$ is $A[K]$-stable for some open subgroup $K$ of $G$.
\end{itemize}

One Typical example to keep in mind is that $W$ is Emerton's completed cohomology introduced in \cite{Eme06} for  arithmetic quotients of  symmetric spaces and $A$ is the Hecke algebra. If these arithmetic quotients are Shimura varieties defined over a number field $F$, one can also take $A=\Z_p[G_F]$.

Let $K$ be an open  subgroup of $G$  sufficiently small so that $W^o$ is $K$-stable and it makes sense to talk about analytic functions on it, cf.  Theorem 27.1 of \cite{Sch11}.  We denote by $W^{K-\an}\subseteq W$ the subspace of $K$-analytic vectors. It is a $\Q_p$-Banach space and an $A$-module.

\begin{thm} \label{ocla}
The action of $A$ on $W^{K-\an}$ is locally analytic.
\end{thm}

\begin{proof}
We denote by $\mathscr{C}^{\an}(K,\Q_p)$ the space of $\Q_p$-valued analytic functions on $K$ with  the unit open ball $\mathscr{C}^{\an}(K,\Q_p)^o$ . Fix $n\geq 1$. Then  
\[W^{K-\an,o}=(W^o\widehat\otimes_{\Z_p} \mathscr{C}^{\an}(K,\Q_p)^o)^K,\]
cf. \cite[2.1]{Pan20}.  The completed tensor product is $p$-torsion free. Hence there is an inclusion
\[W^{K-\an,o}/p^n\subseteq (W^o\otimes_{\Z_p} \mathscr{C}^{\an}(K,\Q_p)^o/p^n)^K.\]
Note that $ \mathscr{C}^{\an}(K,\Q_p)^o/p^n$ is fixed by some open subgroup $K'$ of $K$: when $n=1$, this is  \cite[Lemma 2.1.2.]{Pan20}. The same argument works for any $n$. Hence
\[ W^{K-\an,o}/p^n\subseteq (W^o\otimes_{\Z_p} \mathscr{C}^{\an}(K,\Q_p)^o/p^n)^{K'}=(W^o/p^n)^{K'}\otimes_{\Z_p/p^n} \mathscr{C}^{\an}(K,\Q_p)^o/p^n.\]
(Implicitly we use that $\mathscr{C}^{\an}(K,\Q_p)^o/p^n$ is flat over $\Z_p/p^n$.) Note that all the maps are $A$-equivariant. Thus the image of $A\to\End(W^{K-\an,o}/p^n)$ factors through the image of $A\to \End((W^o/p^n)^{K'})$. But $(W^o/p^n)^{K'}$ is finite by the admissibility. By definition, this means that the action of $A$ on $W^{K-\an}$ is locally analytic.
\end{proof}

\begin{rem}
When $W$ is the completed cohomology of modular curves and $A=\Z_p[G_{\Q_p}]$, this result  implies that  Sen's theory can be applied to $W^{K-\an}$. For example it follows that the Sen operator acts on $W^{K-\an}\widehat\otimes_{\Q_p} C$, cf. \cite[Remark 5.1.16]{Pan20}.
\end{rem}

\bibliographystyle{amsalpha}

\bibliography{bib}

\end{document}